\documentclass[12pt, twoside]{amsart}
\usepackage{graphicx}
\usepackage{amsfonts}
\usepackage{amsthm}
\usepackage{amsmath}
\usepackage{enumerate}
\usepackage{a4wide}

\newcommand{\brc}[1]{\left\lbrace #1 \right\rbrace}

\usepackage{color}
\definecolor{Red}{rgb}{1,0,0}

\newcommand{\Nat}{\mathbb{N}}

\newcommand{\Real}{\mathbb{R}}

\newcommand{\To}{\rightarrow}

\newtheorem{thm}{Theorem}

\theoremstyle{definition}
\newtheorem{defn}[thm]{Definition}
\newtheorem{rem}[thm]{Remark}
\newtheorem{exmp}[thm]{Example}

\newcommand{\upor}{\mathcal{T}}

\begin{document}

\title[Evolutionary Games on Graphs and Discrete Dynamical Systems]{Evolutionary Games on Graphs  \\ and Discrete Dynamical Systems}
\author{Jeremias Epperlein}
\address{Technische Universit\"{a}t Dresden, Department of Mathematics, Center for Dynamics \& Institute for Analysis, 01062 Dresden, Germany}
\email{jeremias.epperlein@tu-dresden.de}

\author{Stefan Siegmund}
\address{Technische Universit\"{a}t Dresden, Department of Mathematics, Center for Dynamics \& Institute for Analysis, 01062 Dresden, Germany}
\email{stefan.siegmund@tu-dresden.de}

\author{Petr Stehl\'\i k}
\address{University of West Bohemia, Faculty of Applied Sciences, Univerzitn\'\i\ 22, 30614 Plze\v{n}, Czech Republic}
\email{pstehlik@kma.zcu.cz}

\begin{abstract}
Evolutionary games on graphs play an important role in the study of evolution of cooperation in applied biology. Using rigorous mathematical concepts from a dynamical systems and graph theoretical point of view, we formalize the notions of attractor, update rules and update orders. We prove results on attractors for different utility functions and update orders. For complete graphs we characterize attractors for synchronous and sequential update rules. In other cases (for $k$-regular graphs or for different update orders) we provide sufficient conditions for attractivity of full cooperation and full defection. We construct examples to show that these conditions are not necessary. Finally, by formulating a list of open questions we emphasize the advantages of our rigorous approach.


 \end{abstract}

\maketitle

\section{Introduction}
Game theory was developed in the 1940s as a mathematical tool to study interactions and decisions of rational agents \cite{aNash}. For a long time it was mainly applied in economics (see e.g.\ \cite{bDixit, bMyerson} and the references therein). In the 1970s, it was introduced to biology via the concept of biological fitness and natural selection \cite{aMaynard} and the evolutionary game theory enabled to study infinite homogeneous populations via replicator equations  \cite{aHofbauer, bHofbauer}. Recently, dynamics in populations which are finite and spatially structured (the structure being represented by graphs) attracted a lot of attention \cite{aNowak, aSzabo}. Evolutionary game theory on graphs has shown that the rate of cooperation strongly depends on the structure of the underlying interaction graph (see e.g.\ \cite{aHauert, aLieb, aNowak, aOhtsuki, aOhtsuki2, aSzabo} and numerous following papers). At the same time, the similar problem of equilibria selection in cooperation games via timing structures has attracted a lot of attention, both in microeconomics (e.g.\ \cite{aKandori, aLagunoff}), as well as in macroeconomics (e.g.\ \cite{aLibich, aLibich2}).

Mathematically, evolutionary games on graphs are very complex structures which bring together notions from graph theory, game theory, dynamical systems and stochastic processes. Mathematical techniques which are used are therefore complex and include complex approximative techniques like pair and diffusion approximation and voter model perturbations \cite{aAllen, aChen, bCox} or are limited to special classes of graphs, e.g., cycles \cite{aOhtsuki}, stars \cite{aBroom, aBroom2} and vertex-transitive graphs \cite{aAllenNowak, aDebarre}. Whereas these papers focus mainly on stochastic updating of vertices, the goal of this paper is to study evolutionary games on graphs with various deterministic update rules in the framework of discrete dynamical systems and formalize the notions of (autonomous and nonautonomous) evolutionary games on graphs, attractors, update rules and update orders. Being aware that this approach may look too formal for some, we believe that this approach should  help to (i) disclose and answer interesting questions (some of them listed in Section \ref{s:oq}), (ii) understand patterns by means of simple examples and counterexamples, (iii) describe analytically dynamics on small graphs (motivated e.g.\ by small sizes of agents in  cooperation games in macroeconomics \cite{aLibich2}) and (iv) bridge a gap between three seemingly separate mathematical areas: graph theory, game theory and dynamical systems. 

In Section \ref{s:egame}, we introduce, for a given graph and an arbitrary utility function, a new formal notion of an (autonomous) evolutionary game on a graph (Definition \ref{d:evgame}) and attractors (Definition \ref{d:attractor}). In our rigorous approach, we are trying to follow the spirit of \cite{aHauert, aLieb, aOhtsuki, aOhtsuki2} and \cite{aSzabo}. In Section \ref{s:gt} we introduce two basic utility functions and relate them to the underlying game-theoretical parameters. In Section \ref{s:Kn} we characterize the attractivity of full defection (Theorem \ref{t:Kn:0}) and full cooperation (Theorem \ref{t:Kn:1}) on complete graphs. On $k$-regular graphs we provide only a sufficient condition for attractivity of these states (Theorem \ref{t:kreg}) and show (in Example \ref{x:Cayley}) that it is not necessary. Section \ref{s:nonautonomous} is devoted to the extension of the notion of an evolutionary game on a graph to the realistic situation that the vertices are not all updated at each time step (Definition \ref{d:update:order}). The new notion of nonautonomous evolutionary game on a graph (Definition \ref{d:nagame}) has the structure of a general nonautonomous dynamical system (Definition \ref{d:nads}). We also introduce attractors and their basins for nonautonomous evolutionary games (Definitions \ref{d:na:attractor} and \ref{d:nabasin}) and relate them to the autonomous case (Remark \ref{r:consistency}). In Section \ref{s:Kn:na} we provide conditions for attractivity of full defection and full cooperation of nonautonomous evolutionary games. The conditions are sufficient for non-omitting update orders but also necessary if a sequential update order is considered (Theorem \ref{t:Kn:0:na}). Example \ref{x:Hella:2} shows that in general the conditions are not necessary. In Section \ref{s:update:order} we provide a complete characterization of attractors of evolutionary games on complete graphs for synchronous (Theorem \ref{t:hd:a}) and sequential (Theorem \ref{t:hd:na}) update orders. In particular, Theorem \ref{t:hd:na}(c) (together with Example \ref{x:cycle}) shows the existence of an attractive cycle. In Section \ref{s:wheels} we discuss the role of different utility functions on irregular graphs. Finally, Section \ref{s:oq} is devoted to concluding remarks and open questions.
\section{Evolutionary games on graphs}\label{s:egame}
We recall some basic definitions from graph theory, see e.g.\ \cite{aGodsil}. A \emph{graph} $G=(V,E)$ is a pair
consisting of a set of \emph{vertices} $V$ and a set of \emph{edges} $E \subseteq \{ e \subseteq V : |e| = 2\}$.
Let $N_k(i)$ denote the \emph{$k$-neighbourhood} of vertex $i\in V$, i.e.\ all vertices with the distance exactly $k$ from $i$. Furthermore, let us define
\[
N_{\leq k}(i) = \bigcup_{j=0}^{k} N_j(i).
\]



We utilize basic concepts
from dynamical systems theory (see e.g.\ \cite{2003:Devaney}).

\begin{defn}\label{d:autonomous-dyna}
Taking into account that our state space is discrete, we strip off any topological
properties and call a map
\[
  \varphi : \Nat_0 \times X \rightarrow X
\]
on an arbitrary set $X$ a \emph{dynamical system} with discrete \emph{one-sided time} $\Nat_0$, if it satisfies the \emph{semigroup property}
\[
  \varphi(0,x) = x
  \quad \textup{and} \quad
  \varphi(t, \varphi(s, x)) = \varphi(t+s, x)
  \quad \textup{for all }
  t, s \in \Nat_0, x \in X  .
\]
The set $X$ is called the \emph{state space} and $\varphi(\cdot) = \varphi(1,\cdot) : X
\rightarrow X$ the corresponding \emph{time-1-map}.
\end{defn}

Any map $\varphi : X \rightarrow X$ is the time-1-map of an induced dynamical system with discrete
one-sided time $\Nat_0$ which is defined via composition or iteration as
follows
\[
  \Nat_0 \times X \rightarrow X
  , \quad
  (t,x) \mapsto \varphi^t(x)
  ,
\]
where $\varphi^0 = \mathrm{id}$ and $\varphi^t = \varphi \circ \cdots \circ \varphi$. In
this paper we describe dynamical systems via their time-1-map.

We will define evolutionary games on graphs as special dynamical systems with the property that a vertex follows a strategy in its 1-neighbourhood which currently yields the highest utility. To prepare this definition we introduce for a ``function on a graph'' the size of the neighbourhood on which its values depend.

\begin{defn}
Let $M, S$ be arbitrary sets and $G=(V,E)$ a graph. We say that a function $f:S^V\rightarrow M^V$ has a \emph{dependency radius} $r \in \Nat_0$ on $G$ if all values $f_i(x)$ of a component $f_i$ of $f$ depend on a component $x_j$ of the argument $x = (x_1,\dots,x_{|V|})$ only if the vertex $j \in V$ is in the $r$-neighbourhood of the vertex $i \in V$, i.e.\ if for each $x,y\in S^V$ and $i,j\in V$ the following implication holds:
\[
  \left\langle \forall k\in V\setminus \lbrace j\rbrace : x_k=y_k \right\rangle 
  \mathrm{ and } 
  \left\langle f_i(x)\neq f_i(y) \right\rangle 
  \;\Rightarrow\; 
  j\in N_{\leq r}(i).
\]
\end{defn}

We are now in a position to formulate evolutionary games on graphs with unconditional imitation update rule  as a dynamical system. This update rule goes back to \cite{aNowak} and is also called ``imitate-the-best''. The basic idea is that at each time step every player determines his utility based on his strategy and the strategies of his neighbours. Based on that, each player adopts the strategy of his neighbour with the highest utility (if it is greater than his own).

\begin{defn}\label{d:evgame}
Let $S$ be an arbitrary set. An \emph{evolutionary game} on $G=(V,E)$ consists of the following two ingredients:

(i) a \emph{utility function} on $G=(V,E)$, i.e.\ a function $u:S^V\rightarrow \Real^V$ which has dependency radius $1$ on $G$,

(ii) the \emph{dynamical system} $\varphi:S^V\rightarrow S^V$ with $\varphi_i:=\mathrm{proj}_i \circ \varphi : S^V\rightarrow S$ given by
\begin{equation}\label{e:imitation}
\varphi_i(x) = \begin{cases}
x_{\max} & \textrm{if } \left|A_i(x)\right|=1 \textrm{ and } A_i(x)=\lbrace x_{\max} \rbrace,\\
x_i & \textrm{if } \left|A_i(x)\right| > 1,
\end{cases}
\end{equation}
where the set $A_i(x)$ is defined by
\begin{equation}\label{e:setAi}
A_i(x) = \left\lbrace x_k: k\in \mathrm{argmax}\left\lbrace u_j(x): j\in  N_{\leq 1}(i) \right\rbrace \right\rbrace.
\end{equation}
\end{defn}

\begin{rem}\label{r:evgame}
(i) Typically, $S$ represents the set of strategies (e.g., $C$ and $D$) and the vector $x=(x_1,\dots,x_{|V|})\in S^V$ the population state (i.e., the spatial distribution of strategies at a given time).

(ii) The cardinality of $A_i(x)$ in \eqref{e:imitation} is used to ensure that all vertices with the highest utility have the same state. If that is not the case, the vertex preserves its current state. Obviously, this may not be a reasonable approach once the set of strategies $S$ contains more than two strategies and the current state may yield worse payoff than all other strategies.

(iii) An evolutionary game $\varphi$ on $G$ has dependency radius $2$ on $G$.

(iv) Whereas the notion of evolutionary games in biological applications (see e.g.\ \cite{aOhtsuki}) sometimes has a probabilistic aspect, our Definition \ref{d:evgame} of an evolutionary game is deterministic. Note that, although we do not follow that direction in this paper, in principle it is possible to extend Definition \ref{d:evgame} to also incorporate that a vertex follows strategies in its 1-neighbourhood with a certain probability.

(v) Given the fact, that each vertex mimics the state of its neighbour with the highest utility, we speak about  \emph{imitation dynamics}. Preserving the deterministic nature, there are various possibilities of defining the dynamics. For example, instead of imitation dynamics $
\varphi^I:=\varphi$ in \eqref{e:imitation}, we could consider \emph{deterministic death-birth} dynamics, in which only the vertices with the lowest utility adopt the state of its neighbour with highest utility, i.e.
\[
\varphi^{DB}_i(x) = \begin{cases}
x_{\max} & \textrm{if } u_i(x)=\min\limits_{j\in V} u_j(x), \left|A_i(x)\right|=1 \textrm{ and } A_i(x)=\lbrace x_{\max} \rbrace ,\\
x_i & \textrm{if } u_i(x)>\min\limits_{j\in V} u_j(x) \textrm{ or }\left|A_i(x)\right|>1.
\end{cases}
\]
Alternatively, we could consider \emph{deterministic birth-death} dynamics in which only vertices in the neighbourhood of vertices with the highest utility are updated, i.e.
\[
\varphi^{BD}_i(x) = \begin{cases}
x_{\max} & \textrm{if } \left|A_i(x)\right|=1,  \max\limits_{j\in N_1(i)} u_j(x)=\max\limits_{j\in V} u_j(x) \textrm{ and } A_i(x)=\lbrace x_{\max} \rbrace ,\\
x_i & \textrm{if } \left|A_i(x)\right| > 1  \textrm{ or }  \max\limits_{j\in N_1(i)} u_j(x)<\max\limits_{j\in V} u_j(x).
\end{cases}
\]
We leave the analysis of such dynamical systems for further research and focus on evolutionary games with imitation dynamics \eqref{e:imitation}.
\end{rem}

We define a notion of distance $d$ in the state space $S^V$ by $d:S^V \times S^V \rightarrow \Real$ with $S\subset\Real$ by
\[
d(x,y):= \left|\lbrace{i\in V : x_i \neq y_i \rbrace} \right|.
\]

Similarly, the distance of a state $x\in S^V$ from a set of states $B\subset S^V$ will be denoted by $\mathrm{dist}: S^V \times 2^{S^{V}} \rightarrow \Real$ and defined by
\[
\mathrm{dist}(x,B) :=\mathrm{inf} \left\lbrace d(x,y): y\in B \right\rbrace.
\]

\begin{defn}\label{d:attractor}
Let $\varphi:S^V\rightarrow S^V$ be an evolutionary game on $G=(V,E)$ and $A\subset S^V$ \emph{invariant} under $\varphi$, i.e.\ $\varphi(A) = A$. Then $A \neq \emptyset$ is called \emph{attractor} of $\varphi$ if for any $x\in S^V$ with $\mathrm{dist}(x,A)\leq 1$ there exists $t\geq 0$ such that $\varphi^t(x)\in A$.

If $A=S^V$, we say that $A$ is the \emph{trivial attractor}, otherwise $A$ is said to be a \emph{nontrivial attractor}.
\end{defn}

\begin{rem}
If $S$ and $V$ are finite sets, then $d$ generates the discrete topology on $S^V$. With this topology an invariant set $A$ is an attractor if and only if $ \lim\limits_{t\rightarrow\infty} \mathrm{dist}\left( \varphi^t(x),A \right)=0$ for all $x\in S^V$ with $\mathrm{dist}(x,A)\leq 1$.
\end{rem}

%
\section{Utility Function and Cooperative Games}\label{s:gt}

We discuss cooperative evolutionary games (see e.g.\ \cite{bNowak}) with utilities which are implied by a static game (see e.g.\ \cite{bDixit} for an introduction to game theory) on a state space $S=\lbrace C,D \rbrace$ consisting of two strategies, each vertex can either cooperate ($C$) or defect ($D$).
\begin{center}
\begin{tabular}{c|cc}
& C & D \\ \hline
C & $a$ & $b$ \\
D & $c$ & $d$ \\
\end{tabular}
\end{center}
Consequently, a player gets $a$ if both he and his partner cooperate, he gets $b$ if he cooperates and his partner defects, if he defects and his partner cooperates he gets $c$, if both players defect, he gets $d$.
We focus on cooperation games and therefore make the following assumptions on the parameters $a,b,c,d$:
\begin{itemize}
	\item[(A1)] For the sake of brevity, we assume that no two parameters are equal.
	\item[(A2)] It is always better if both players cooperate than if they both defect, i.e.\ $a>d$.
	\item[(A3)] If only one cooperates, it is more advantageous to be the defector, i.e.\ $c>b$.
	\item[(A4)] No matter what strategy a player chooses, it is always better for him if his opponent cooperates, i.e.\ $a>b$ and $c>d$.
	\item[(A5)] $a,c$ are positive, i.e.\ there is a positive reward for cooperation.
\end{itemize}

\begin{defn}
We say that a parameter vector $(a,b,c,d)$ is \emph{admissible} if it satisfies assumptions (A1)-(A5). The set $\mathcal{P} \subset \Real^4$ of all such quadruplets is called the \emph{set of admissible parameters}.
\end{defn}

The set $\mathcal{P}$ of admissible parameters splits into four regions corresponding to scenarios which we call Full cooperation, Hawk and dove, Stag hunt\footnote{Strictly speaking, the Stag hunt scenario is usually considered with $a>d>c>b$, we use this notation, since the equilibria share the same structure.} and Prisoner's dilemma:
\begin{align*}
\mathcal{P}_{\mathrm{FC}} &= \lbrace{(a,b,c,d)\in\Real^4: a>c>b>d \rbrace}, \\
\mathcal{P}_{\mathrm{HD}} &= \lbrace{(a,b,c,d)\in\Real^4: c>a>b>d \rbrace}, \\
\mathcal{P}_{\mathrm{SH}} &= \lbrace{(a,b,c,d)\in\Real^4: a>c>d>b \rbrace}, \\
\mathcal{P}_{\mathrm{PD}} &= \lbrace{(a,b,c,d)\in\Real^4: c>a>d>b \rbrace}.
\end{align*}
Obviously,
\begin{equation}\label{e:P}
\mathcal{P} = \mathcal{P}_{\mathrm{FC}} \cup \mathcal{P}_{\mathrm{HD}} \cup \mathcal{P}_{\mathrm{SH}} \cup \mathcal{P}_{\mathrm{PD}}.
\end{equation}
In the case of static games, a (mixed) Nash equilibirum for two player games is a pair of mixed strategies $(\sigma_1^*,\sigma_{2}^*)$ such that
\begin{align*}
u_1(\sigma_1^*,\sigma_{2}^*) &\geq (\sigma_1,\sigma_{2}^*), \quad \textrm{for all } \sigma_1\in\Sigma_1,\\
u_2(\sigma_1^*,\sigma_{2}^*) &\geq (\sigma_1^*,\sigma_{2}), \quad \textrm{for all } \sigma_2\in\Sigma_2,
\end{align*}
where $\Sigma_i$ denotes the set of all mixed strategies of player $i$ (see \cite{bDixit} for more details). In the admissible regions the Nash equilibria have the following structure.

\begin{center}
\begin{tabular}{ccccc}
 & abbr. & scenario & Nash equilibria \\[0.5ex] \hline
$a>c>b>d$ & FC & Full cooperation & (C,C) \\
$c>a>b>d$ & HD & Hawk \& Dove & (C,D), (D,C) and a mixed equilibrium \\
$a>c>d>b$ & SH & Stag hunt & (C,C), (D,D) and a mixed equilibrium \\
$c>a>d>b$ & PD & Prisoner's dilemma & (D,D)
\end{tabular}
\end{center}

Note that the mixed Nash equilibirum $\frac{d-b}{(a-c)+(d-b)}$ provides the minimal payoff in the SH game and the maximal payoff in the HD game. As we will see, this fact influences the stability of corresponding interior points of, e.g., replicator dynamics, see \cite{bHofbauer}. 

There are two natural ways how to define utilities on evolutionary graphs with $S=\lbrace 0,1 \rbrace$ (the state $1$ corresponds to $C$ and the state $0$ to $D$). We could either consider the \emph{aggregate utility}
\begin{equation}\label{e:utility}
u_i^A(t)=a\sum\limits_{j\in N_1(i)} s_is_j + b \sum\limits_{j\in N_1(i)} s_i(1-s_j) + c\sum\limits_{j\in N_1(i)} (1-s_i)s_j+ d\sum\limits_{j\in N_1(i)} (1-s_i)(1-s_j),
\end{equation}
or the \emph{mean utility}
\begin{equation}\label{e:utility:M}
u_i^M(t)=\frac{1}{\vert N_1(i)\vert}u_i^A(t).
\end{equation}

Note, that on regular graphs $\vert N_1(i) \vert$ is constant and both utilities yield the same evolutionary games. However, on irregular graphs, this is no longer true (see Section \ref{s:wheels}) and one could argue which utility is more realistic.

\begin{figure}
\begin{minipage}{0.4\textwidth}
\begin{center}
\includegraphics[width=\textwidth]{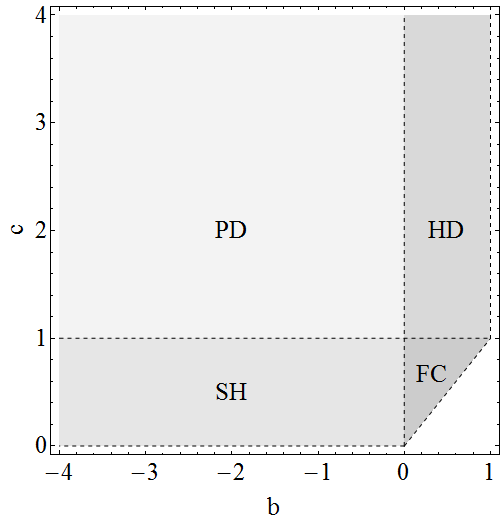}\\
\end{center}
\end{minipage}
\caption{Set of admissible parameters and four game-theoretic scenarios, $a=1$ and $d=0$.}\label{f:scen}
\end{figure}

\begin{rem}\label{r:norm}
Once we consider the mean utility function (or the aggregate utility function on regular graphs) we could, without loss of generality, normalize parameters $a,b,c,d$ so that $\tilde{a}=1$ and $\tilde{d}=0$ by the following map
\[
\tilde{x}=\frac{x-d}{a-d},\quad x=a,b,c,d.
\]
Conversely, given normalized values of parameters $\tilde{a},\tilde{b},\tilde{c},\tilde{d}$, we could for arbitrary $a$ and $d$ such that $a>d$ construct non-normalized values of parameters by
\[
x= d+(a-d)\tilde{x},\quad x=a,b,c,d.
\]

This allows us to simplify conditions or plot regions corresponding to various scenarios, cf. Figure \ref{f:scen} where four scenarios are depicted.
\end{rem}

\section{Evolutionary Games on $K_n$}\label{s:Kn}
In this section we consider evolutionary games generated by the simplest graphs -- complete graphs $K_n$ and regular graphs. Note that evolutionary games on $K_n$ correspond to dynamics of a well-mixed (nonspatial) population. We focus on the attractivity of full defection $(0,0,\ldots,0)$ and full cooperation $(1,1,\ldots,1)$ and its connection to parameters $a,b,c,d$.
\begin{thm}\label{t:Kn:0}
For all admissible $(a,b,c,d)\in \mathcal{P}$ the following two statements are equivalent:
\begin{enumerate}[\upshape (a)]
\item $\lbrace (0,0,\ldots,0)\rbrace$ is an attractor of the evolutionary game $\varphi$ on $K_n$ with the utility function \eqref{e:utility}.
\item $(a,b,c,d)\in \mathcal{P}$ satisfy
	\begin{equation}\label{e:Kn:0}
		b<d \quad \mathrm{or} \quad n<1+\frac{c-d}{b-d}.
	\end{equation}
\end{enumerate}
\end{thm}

\begin{proof}
Choose $x\in S^V=\lbrace{0,1\rbrace}^V$ such that $d \left( x, (0,0,\ldots,0)\right) = 1$. Then there exists a unique $i\in V$ such that $x_i=1$ and $x_j=0$ for all $j\neq i$ (i.e.\ $i$ is the unique cooperator). Consequently, the utilities are:
\begin{align*}
u_i(x) &= (n-1)b, \\
u_j(x) &= c+ (n-2)d.
\end{align*}

$(b) \Rightarrow (a)\;$
The inequalities \eqref{e:Kn:0} imply that
\[
u_i(x)-u_j(x) = (n-1)b - c - (n-2)d = n (b-d) - (b-d) - (c-d) < 0.
\]
Therefore, $u_j(x)>u_i(x)$ for all $j\neq i$. Hence, $\varphi_j(x)=0$ for all $j\in V$ and $\varphi(x)=(0,0,\ldots,0)$.

$(a) \Rightarrow (b)\;$
Assume that \eqref{e:Kn:0} does not hold, i.e.\footnote{Note, that $b=d$ is not admissible.}
\[
b>d \quad \mathrm{and} \quad n \geq 1+\frac{c-d}{b-d}.
\]
If $n=1+\frac{c-d}{b-d}$, then $u_i(x)=u_j(x)$. This implies that $\varphi_j(x)=x_j$ and $\varphi(x)=x$. If $n>1+\frac{c-d}{b-d}$, then $u_i(x)>u_j(x)$. This implies that $\varphi_j(x)=1$ and $\varphi(x)=(1,1,\ldots,1)$, which implies that $(0,0,\ldots,0)$  cannot be reached from $x$, since  $\varphi((1,1,\ldots,1))=(1,1,\ldots,1)$.
\qedhere
\end{proof}

A similar result could be obtained for the full cooperation state $(1,1,\ldots,1)$.

\begin{thm}\label{t:Kn:1}
For all admissible $(a,b,c,d)\in \mathcal{P}$ the following two statements are equivalent:
\begin{enumerate}[\upshape (a)]
\item $\lbrace (1,1,\ldots,1)\rbrace$ is an attractor of the evolutionary game $\varphi$ on $K_n$ with the utility function \eqref{e:utility}.
\item $(a,b,c,d)\in \mathcal{P}$ satisfy
	\begin{equation}\label{e:Kn:1}
		a>c \quad \mathrm{and} \quad n>1+\frac{a-b}{a-c}.
	\end{equation}
\end{enumerate}
\end{thm}

\begin{proof}
The proof is very similar to the proof of Theorem \ref{t:Kn:0}.

$(b) \Rightarrow (a)\;$ 
Choose $x\in S^V=\lbrace{0,1\rbrace}^V$ such that $d \left( x, (1,1,\ldots,1)\right) = 1$. Then there exists a unique $i\in V$ such that $x_i=0$ and $x_j=1$ for all $j\neq i$ (i.e.\ $i$ is the unique defector). Consequently, the utilities are:
\begin{align*}
u_i(x) &= (n-1)c, \\
u_j(x) &= b+ (n-2)a.
\end{align*}
Then \eqref{e:Kn:1} implies that
\[
u_i(x)-u_j(x) = (n-1)c - b - (n-2)a = -n(a-c)+(a-c)+(a-b) < 0.
\]
Therefore, $u_j(x)>u_i(x)$ for all $j\neq i$. Hence, $\varphi_j(x)=1$ for all $j\in V$ and $\varphi(x)=(1,1,\ldots,1)$.

$(a) \Rightarrow (b)\;$
Assume that \eqref{e:Kn:1} does not hold, i.e.
\[
a<c \quad \mathrm{or} \quad n \leq 1+\frac{a-b}{a-c}.
\]
If the equality $n = 1+\frac{a-b}{a-c}$ holds, then $\varphi(x)=x$. Otherwise, $\varphi(x)=(0,0,\ldots,0)$. Consequently, $(1,1,\ldots,1)$ is not attractive.
\qedhere
\end{proof}

\begin{figure}

\begin{minipage}{\textwidth}
\begin{center}
\begin{minipage}{0.45\textwidth}
\begin{center}
\includegraphics[width=\textwidth]{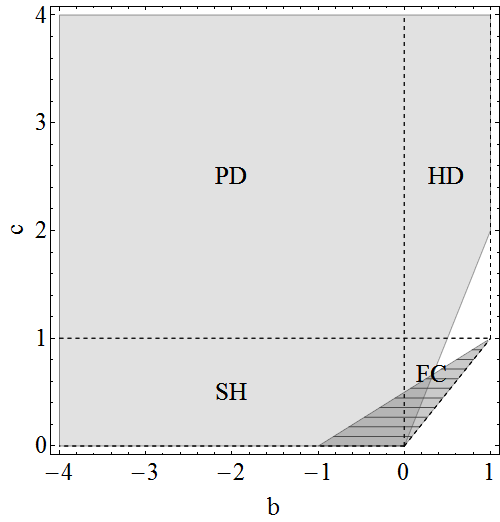}
\end{center}
\end{minipage}
\begin{minipage}{0.45\textwidth}
\begin{center}
\includegraphics[width=\textwidth]{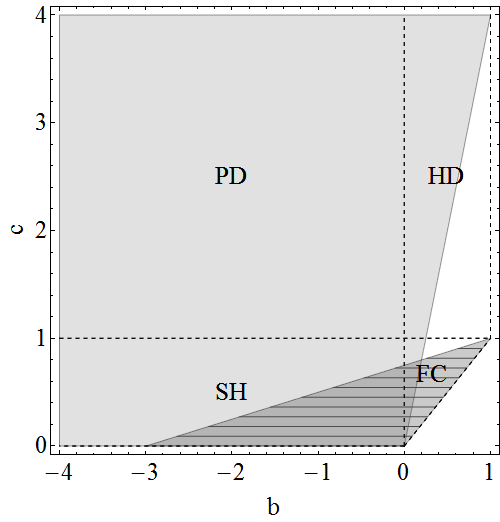}
\end{center}
\end{minipage}
\end{center}
\end{minipage}
\caption{Attractivity regions of $(0,0,\ldots,0)$ (light gray) and $(1,1,\ldots,1)$ (horizontally hatched) ($a=1$ and $d=0$ are fixed, cf.\ Remark \ref{r:norm}). (a) $n=3$ on the left and (b) $n=5$ on the right.}\label{f:att:reg}

\end{figure}

\begin{rem}
Theorems \ref{t:Kn:0} and \ref{t:Kn:1} show that attractivity of full defection and full cooperation is qualitatively different for admissible parameters in the four different regions FC, HD, SH, PD in \eqref{e:P}. In the PD case, there is no dependence on $n$ and the values of the parameters. On the other hand, in the other cases different behaviour could occur for different parameter values and in dependence of the size of the graph $n$. Evidently, the richest situation occurs in the FC case in which it is possible that, depending on the parameter, full defection and full cooperation, as well as none of them, or both of them together are attractive, cf.\ Table \ref{tab:k:Kn} and Figure \ref{f:att:reg}.

\begin{table}[h]
\begin{tabular}{ccccc}
& PD & HD & SH & FC \\ \hline
$(0,0,\ldots,0)$ attractive & always &$n<1+\frac{c-d}{b-d}$  & always & $n<1+\frac{c-d}{b-d}$ \\
$(1,1,\ldots,1)$ attractive & never & never & $n>1+\frac{a-b}{a-c}$ & $n>1+\frac{a-b}{a-c}$\\
\end{tabular}
\vspace*{1ex}
\caption{$K_n$ - Connection between graph size, different scenarios and attractivity of states $(0,0,\ldots,0)$ and $(1,1,\ldots,1)$}.\label{tab:k:Kn}
\end{table}

We observe that in the PD case, $(0,0,\ldots,0)$ is always attractive. Similarly, in the FC case $(1,1,\ldots,1)$ is the unique attractor if and only if $n > 1+ \max\left\lbrace \frac{c-d}{b-d}, \frac{a-b}{a-c} \right\rbrace$.

Finally, note that the results are consistent with those for infinite populations studied in the evolutionary game theory \cite{bHofbauer, bNowak}, in which the full cooperation is ESS if and only if $a>c$ and the full defection is ESS if and only if $b<d$ (those inequalities are obtained as $n\rightarrow\infty)$. Note that for finite complete graphs additional conditions involving the size $n$ are involved (see \eqref{e:Kn:0}-\eqref{e:Kn:1}). Thus, the parameter region in finite populations is larger for full defection and smaller for full cooperation, see Figure \ref{f:att:reg}.
\end{rem}

The ideas from the proofs of Theorems \ref{t:Kn:0} and \ref{t:Kn:1} easily carry over to $k$-regular graphs.
\begin{thm}\label{t:kreg}
Let $(a,b,c,d)\in \mathcal{P}$ be admissible and $G$ be a $k$-regular graph, $k\geq 2$. Then:
\begin{enumerate}[\upshape (a)]
	\item If
		\begin{equation}\label{e:kreg:0}
			k(b-d)<c-d,
		\end{equation}
		holds, then $\lbrace (0,0,\ldots,0)\rbrace$ is an attractor of the evolutionary game $\varphi$ on $G$ with the utility function \eqref{e:utility}.
	\item If
		\begin{equation}\label{e:kreg:1}
			k(a-c)>a-b,
		\end{equation}
		holds, then $\lbrace (1,1,\ldots,1)\rbrace$ is an attractor of the evolutionary game $\varphi$ on $G$ with the utility function \eqref{e:utility}.
\end{enumerate}
\end{thm}
However, the following example shows that \eqref{e:kreg:0} and \eqref{e:kreg:1} are only sufficient and not necessary for the attractivity of $\lbrace (0,0,\ldots,0)\rbrace$ and $\lbrace (1,1,\ldots,1)\rbrace$ for evolutionary games on $k$-regular graphs.

\begin{exmp}\label{x:Cayley}
Consider the undirected Cayley graph (see \cite[p.\ 34]{aGodsil}) of the Dihedral group of order 24 (as a permutation group on $\{1,\dots,12\}$,
see \cite[p.\ 46]{aHungerford} for the notation)
with generators
\begin{align*}
  g_1&=(2\; 12)(3\; 11)(4\; 10)(5\; 9)(6\; 8), \\
  g_2&=(1\; 2)(3\; 12)(4\; 11)(5\; 10)(6\; 9)(7\; 8), \\
  g_3&=(1\; 4)(2\; 3)(5\; 12)(6\; 11)(7\; 10)(8\; 9).
\end{align*}
The graph is depicted in Figure \ref{f:Cayley}.
\begin{figure}
\begin{center}
\begin{minipage}{14cm}
\begin{flushleft}
  \includegraphics[scale=0.8]{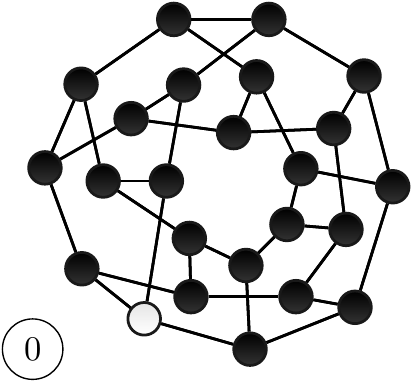}
  \includegraphics[scale=0.8]{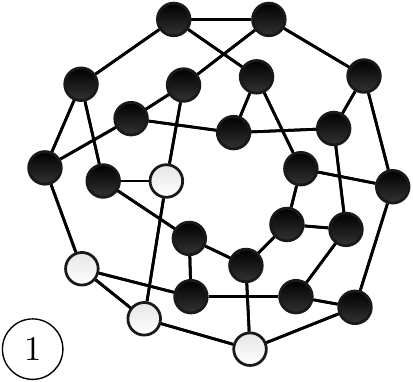}
  \includegraphics[scale=0.8]{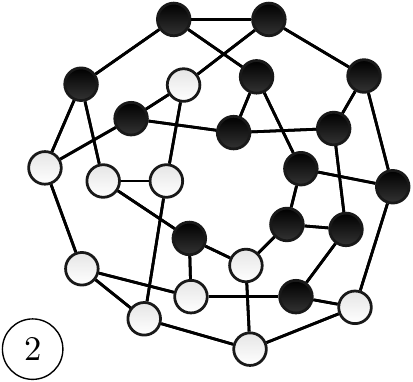}
  \includegraphics[scale=0.8]{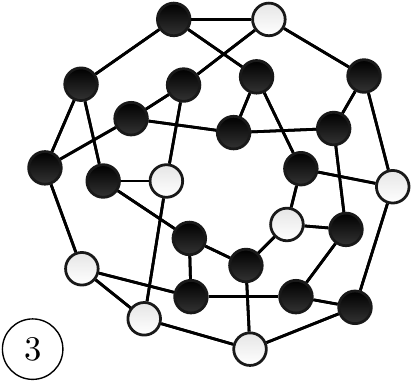}
  \includegraphics[scale=0.8]{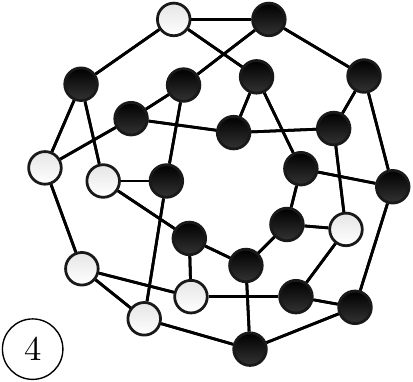}
  \includegraphics[scale=0.8]{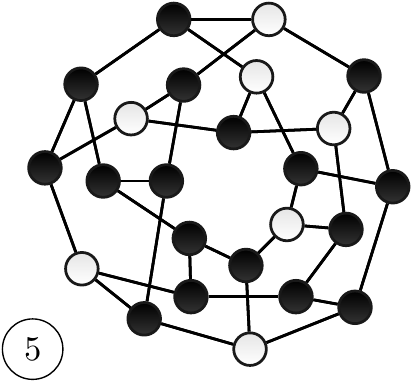}
  \includegraphics[scale=0.8]{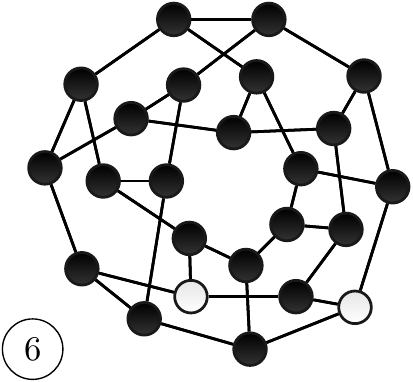}
  \includegraphics[scale=0.8]{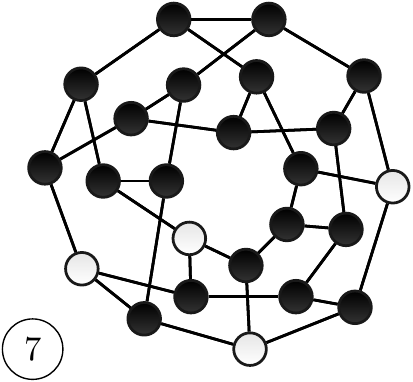}
  \includegraphics[scale=0.8]{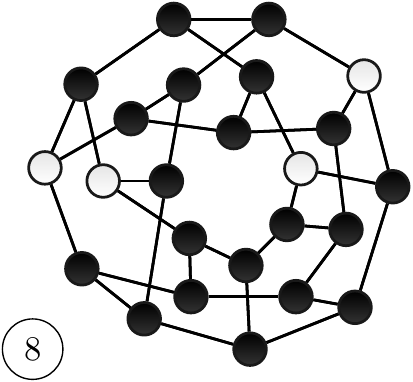}
  \includegraphics[scale=0.8]{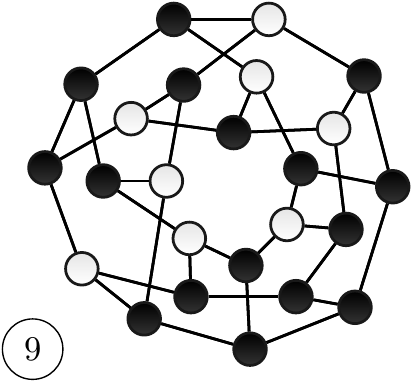}
  \includegraphics[scale=0.8]{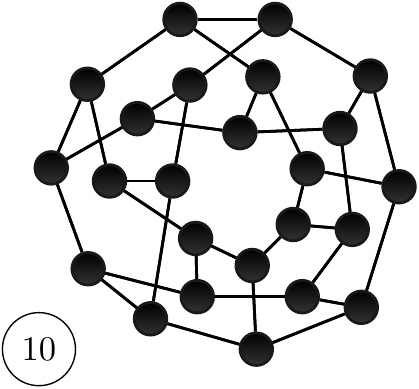}
\end{flushleft}
\end{minipage}

  \caption{A trajectory of the evolutionary game from Example \ref{x:Cayley}. For possible reproduction we provide its graph6 notation (see \cite{wMcKay}):
}\label{f:Cayley}
\verb|WsOPA?OG?[?E@C?o@??@??O?????????s??k?@@_?Cg??KO|.
\end{center}
\end{figure}
Consider the evolutionary game on $G$ with utility function (\ref{e:utility}) and parameters $(a,b,c,d)=(1,0.88,1.74,0)$. Then the
following inequalities are fulfilled:
\begin{align*}
 &0\cdot c+3 \cdot d = 0 &<\quad& 1\cdot c+2\cdot d = 1.74 &<\quad& 0\cdot a+3\cdot b =2.64\\
 < \quad&1\cdot a+2\cdot b =2.76 &<\quad& 2\cdot a+1\cdot b = 2.88 &<\quad& 3\cdot a+0\cdot b = 3\\
 < \quad&2\cdot c+1\cdot d =3.48 &<\quad& 3\cdot c+0\cdot d = 5.22
\end{align*}
Since $G$ is $3$-regular, all other parameters that satisfy the same inequalities will lead to the same
evolutionary game.

The initial state with exactly one cooperator, whose position does not matter because the graph is vertex-transitive,
reaches $(0,0,\dots,0)$ after $10$ steps for these parameters (see Figure~\ref{f:Cayley}). Therefore $(0,0,\ldots,0)$ is attractive, although \eqref{e:kreg:0} is violated.
\end{exmp}

%
\section{General Asynchronous Update - Nonautonomous Evolutionary Games}\label{s:nonautonomous}
In an evolutionary game all nodes are updated in a synchronous way (i.e.\ all nodes at each time step), cf.\ Definition \ref{d:evgame}. To formulate asynchronous update which is only updating a certain subset of nodes at each time step, we introduce a notion of update order.
\begin{defn}\label{d:update:order}
A set-valued function $\upor:\Nat_0 \rightarrow 2^V$ is called \emph{update order}.

The update order $\upor$ is called
\begin{itemize}
\item \emph{non-omitting}, if for each vertex $v\in V$ and each $t_0\in\Nat_0$, there exists $t>t_0$ such that $v\in\upor(t)$; otherwise $\upor$ is called \emph{omitting},
\item \emph{periodic}, if there exists $T\in \Nat$ such that $\upor(t+T)=\upor(t)$ for each $t\in\Nat_0$,
\item \emph{synchronous}, if $\upor(t)=V$ for each $t\in\Nat_0$; otherwise $\upor$ is called \emph{asynchronous},
\item \emph{sequential}, if vertices can be ordered so that $\upor(t) = \lbrace (t+1)(\mathrm{mod}\, n)\rbrace$.
\end{itemize}
\end{defn}

\begin{rem}
Obviously, synchronous and sequential update orders are automatically non-omitting and periodic. A periodic update order could be omitting and a non-omitting update order is not necessarily periodic.
\end{rem}

\begin{exmp}\label{x:Hella}
Let $G$ be a finite graph with $V=\lbrace 1,2,\ldots, n\rbrace$, for some $n\in \Nat$, $n\geq 3$. Then $\upor:\Nat_0 \rightarrow 2^V$ given by
\begin{equation}\label{e:Hella}
\upor(t) = \begin{cases}
\lbrace 1,2 \rbrace				& \textrm{if } t \textrm{ is even},	\\
\lbrace 3,4,\ldots, n \rbrace	& \textrm{if } t \textrm{ is odd},
\end{cases}
\end{equation}
is a periodic and non-omitting update order.
\end{exmp}

Let us define
\[
\left(\Nat_0\right)^2_\geq := \brc{ (t,s) \in \Nat_0 \times \Nat_0: t\geq s}.
\]
\begin{defn}\label{d:nads}
A \emph{nonautonomous dynamical system} (or \emph{two-parameter process}, or \emph{two-parameter semiflow}) $\varphi$ is a map
\[
\varphi: \left(\Nat_0\right)^2_\geq \times M \rightarrow M,
\]
which satisfies the \emph{two-parameter semiflow property}
\[
\varphi(t,t,x)=x,\quad \varphi(t,r,\varphi(r,s,x))=\varphi(t,s,x),
\]
for all $x \in M$ and $t,r,s\in\Nat_0$ such that $t\geq r\geq s$.
\end{defn}

\begin{defn}\label{d:nagame}
A \emph{nonautonomous evolutionary game} on $G=(V,E)$ with a utility function $u$ and update order $\upor:\Nat_0 \rightarrow 2^V$ is the nonautonomous dynamical system $\varphi: \left(\Nat_0\right)^2_\geq \times S^V \rightarrow S^V$ with $\varphi_i:=\mathrm{proj}_i \circ \varphi : \left(\Nat_0\right)^2_\geq \times S^V\rightarrow S$ defined by
\[
\varphi_i(t+1,t,x) = \begin{cases}
x_{\max} & \textrm{if } i\in\upor(t), \left|A_i(x)\right|=1 \textrm{ and } A_i(x)=\{ x_{\max} \},\\
x_i & \mathrm{otherwise},
\end{cases}
\]
where $A_i(x)$ is defined in \eqref{e:setAi}.
\end{defn}

\begin{rem}
Definition \ref{d:nagame} is a nonautonomous version of an evolutionary game with imitation dynamics (cf.\ Remark \ref{r:evgame}). As in the case of (autonomous) evolutionary games, we could consider alternative definitions of  deterministic nonautonomous evolutionary games.
\end{rem}

\begin{defn}\label{d:na:attractor}
Let $\varphi: \left(\Nat_0\right)^2_\geq \times S^V \rightarrow S^V$ be a nonautonomous evolutionary game on $G=(V,E)$. A set $A\subset \Nat_0 \times S^V$ is called \emph{attractor} of $\varphi$ if
\begin{enumerate}[(a)]
\item $A$ is \emph{invariant}, i.e.\ for all $(t,t_0)\in \left(\Nat_0\right)^2_\geq$:
\[
\varphi(t,t_0,A(t_0)) = A(t)
\]
where $A(t) := \{x \in S^V : (t,x) \in A\}$.
\item $A$ is \emph{attracting}, i.e.\ for any $(t_0,x)\in \Nat_0\times S^V$ with $\mathrm{dist}(x,A(t_0))\leq 1$ we have
\[
\lim\limits_{t\rightarrow\infty} \mathrm{dist}\left( \varphi(t,t_0,x),A(t) \right)=0.
\]
\end{enumerate}
\end{defn}

\begin{defn}\label{d:nabasin}
The \emph{basin of attraction} of an invariant set $A\subset \Nat_0\times S^V$ is the set
\[
B(A):= \big\{ (s,x)\in \Nat_0 \times S^V: \lim\limits_{t\rightarrow\infty} \mathrm{dist}\left( \varphi(t,s,x),A(t) \right)=0 \big\}.
\]
\end{defn}

We can make the following simple observation.

\begin{rem}\label{r:basin}
Let $\varphi$ be a nonautonomous evolutionary game and $A\subset \Nat_0\times S^V$ be an invariant set such that $A(t)=A(s)$ for all $t,s\in\Nat_0$. If we define
\[
D_1(A):=\big\{ (t,x) \in \Nat_0 \times S^V: \mathrm{dist}(x,A(t))=1 \big\},
\]
then $A$ is an attractor of a nonautonomous evolutionary game $\varphi: \left(\Nat_0\right)^2_\geq \times S^V \rightarrow S^V$ if and only if $D_1(A) \subset B(A)$.
\end{rem}

\begin{rem}\label{r:consistency}
A nonautonomous evolutionary game $\varphi: \left(\Nat_0\right)^2_\geq \times S^V \rightarrow S^V$ with synchronous update order $\upor(t)=V$ induces an associated (autonomous) evolutionary game $\psi: S^V \rightarrow S^V$ by setting
\begin{equation}\label{e:induced:psi}
	\psi(x) := \varphi(t+1,t,x) \textrm{ for all } x\in S^V,
\end{equation}
for an arbitrary $t\in \Nat_0$. $\psi$ in \eqref{e:induced:psi} is well-defined because synchronous update order implies that $\varphi(t+1,t,x)=\varphi(s+1,s,x)$ for all $t,s\in\Nat_0$ and $x\in S^V$.

An attractor $A \subset \Nat_0 \times S^V$ of $\varphi$ which is time-independent (i.e.\ $A(t)=A(s)$ for all $t,s\in\Nat_0$) induces an attractor $C:=A(t)$ of $\psi$, because $\lim_{t\rightarrow\infty} \mathrm{dist}\left( \varphi(t,t_0,x),A(t) \right)=0$, together with the fact that $\varphi(t,t_0,x)=\psi^{t-t_0}(x)$, implies that
\[
\lim\limits_{t\rightarrow\infty} \mathrm{dist}( \psi^{t-t_0}(x),C )=0.
\]
In this situation the domain of attraction $B(A)$ and the set $D_1(A)$ from Remark \ref{r:basin} are also time-independent and we identify them with the sets $\{x \in S^V : (0,x) \in B(A)\}$ and $\{x \in S^V : (0,x) \in D_1(A)\}$.
In general, an attractor $A$ of a nonautonomous evolutionary game $\varphi$ can be time-dependent, e.g.\ if $A$ is a periodic orbit $A(t+p)=A(t)$ for all $t\in\Nat_0$ and some natural number $p\geq 2$, which is attractive.
If $A\subset \Nat_0 \times S^V$ is a time-independent attractor then we say that $A(0) \subset S^V$ is an attractor. If, moreover, $A=\Nat_0 \times \{ x \}$ then we say that $x$ is an attractor.
\end{rem}


\section{Nonautonomous Evolutionary Games on $K_n$}\label{s:Kn:na}
We follow the ideas from Section \ref{s:Kn} and study the attractivity of full cooperation $(1,1,\ldots,1)$ and full defection $(0,0,\ldots,0)$ on complete graphs for nonautonomous evolutionary games with the utility function \eqref{e:utility}.

\begin{thm}\label{t:Kn:0:na}
Let $\varphi$ be a nonautonomous evolutionary game on $K_n$ with the aggregate utility function \eqref{e:utility}, $(a,b,c,d)\in \mathcal{P}$ and $\upor:\Nat_0 \rightarrow 2^V$ be a non-omitting update order. Then,
\begin{enumerate}[\upshape (i)]
\item if \eqref{e:Kn:0} holds, then $(0,0,\ldots,0)$ is an attractor of $\varphi$,
\item if \eqref{e:Kn:1} holds, then $(1,1,\ldots,1)$ is an attractor of $\varphi$.
\end{enumerate}

Moreover, if $\upor$ is sequential, then the reverse implications also hold.
\end{thm}

\begin{proof}
The proof is very similar to the ideas from the proofs of Theorems \ref{t:Kn:0} and \ref{t:Kn:1}. To prove statement (i), let us assume that \eqref{e:Kn:0} holds and $x\in S^V=\lbrace{0,1\rbrace}^V$ is such that $d \left( x, (0,0,\ldots,0)\right) = 1$. Then there exists a unique $i\in V$ such that $x_i=1$ and $x_j=0$ for all $j\neq i$. The utilities satisfy $u_j(x)>u_i(x)$ for all $j\neq i$. Hence, $\varphi(t+1,t,x)=x$ if $i\notin \upor(t)$ and $\varphi(t+1,t,x)=(0,0,\ldots,0)$ if $i\in \upor(t)$. Since the update rule is non-omitting, we know that there exists $t\in\Nat_0$ such that $i\in\upor(t)$ and consequently for each $s\in\Nat$ we have
\[
\varphi(t+s,0,x)=(0,0,\ldots,0).
\]
The latter statement (ii) is proven in the same way.

Finally, we would like to show that if the update order is sequential, then the conditions \eqref{e:Kn:0} and \eqref{e:Kn:1} are also necessary. Suppose by contradiction that \eqref{e:Kn:0} does not hold. Then for each $x\in S^V=\lbrace{0,1\rbrace}^V$ with $d \left( x, (0,0,\ldots,0)\right) = 1$ we have that for each $t\in\Nat_0$
\begin{equation}\label{e:Kn:0:na:proof}
\varphi(t+1,t,x) \neq (0,0,\ldots,0).
\end{equation}
If $(0,0,\ldots,0)$ is an attractor then there exists $t\in\Nat$ such that
\[
\varphi(t,0,x)\neq(0,0,\ldots,0)\quad \mathrm{and} \quad \varphi(t+1,0,x)=(0,0,\ldots,0),
\]
which implies that
\[
\varphi(t+1,t,x)=(0,0,\ldots,0),
\]
a contradiction to \eqref{e:Kn:0:na:proof}.
\end{proof}

Obviously, if the update order is not non-omitting, the single cooperator need not have a chance to switch and therefore there are no conditions which ensure attractivity of $(0,0,\ldots,0)$ or $(1,1,\ldots,1)$. In the following example we show that for general non-omitting update orders, we cannot reverse the implications, since either \eqref{e:Kn:0} or \eqref{e:Kn:1} need not be necessary.

\begin{exmp}\label{x:Hella:2}
Let us consider a nonautonomous evolutionary game on $K_n$ with utility function \eqref{e:utility} and with non-omitting and periodic update order given by \eqref{e:Hella} in Example \ref{x:Hella}. Let us assume that $(a,b,c,d)\in \mathcal{P}$ are such that \eqref{e:Kn:0} is not satisfied, i.e.
\begin{equation}\label{e:Hella:1c}
(n-1)b>c+(n-2)d,
\end{equation}
but
\begin{align}
a+(n-2)b &<2c+(n-3)d, \label{e:Hella:2c} \\
2a+(n-3)b &< 3c+(n-4)d. \label{e:Hella:3c}
\end{align}

Condition \eqref{e:Hella:1c} implies that one cooperator has a higher utility than $(n-1)$ defectors, inequalities \eqref{e:Hella:2c} and \eqref{e:Hella:3c} ensure that two (or three) cooperators have lower utility than $(n-2)$ (or $(n-3)$) defectors.

If  $x\in S^V=\lbrace{0,1\rbrace}^V$ is such that $d \left( x, (0,0,\ldots,0)\right) = 1$ (i.e.\ there exists a unique $i\in V$ such that $x_i=1$), then we have two possibilities
\begin{itemize}
\item $i\in\lbrace 1,2 \rbrace$. Since $i\in \upor(0)$, we have that
\begin{align*}
\varphi(1,0,x) & \overset{\eqref{e:Hella:1c}}{=} (1,1,0,0,\ldots,0), \\
\varphi(2,0,x) &\overset{\eqref{e:Hella:2c}}{=} (1,1,0,0,\ldots,0), \\
\varphi(3,0,x) &\overset{\eqref{e:Hella:2c}}{=} (0,0,0,0,\ldots,0).
\end{align*}
\item $i\in\lbrace 3,4,\ldots,n \rbrace$. Then without loss of generality, we can assume that $i=3$. Hence,
\begin{align*}
\varphi(1,0,x) & \overset{\eqref{e:Hella:1c}}{=} (1,1,1,0,\ldots,0), \\
\varphi(2,0,x) &\overset{\eqref{e:Hella:3c}}{=} (1,1,0,0,\ldots,0), \\
\varphi(3,0,x) &\overset{\eqref{e:Hella:2c}}{=} (0,0,0,0,\ldots,0).
\end{align*}
\end{itemize}
Consequently, we have shown that $(0,0,\ldots,0)$ is an attractor although \eqref{e:Kn:0} is not satisfied.
\end{exmp}

\section{Existence of Attractors and Update Orders}\label{s:update:order}
In the previous sections we have seen that \eqref{e:Kn:0} and \eqref{e:Kn:1} are necessary and sufficient for the existence of the attractors $(0,0,\ldots,0)$ and $(1,1,\ldots,1)$ on $K_n$ if the synchronous or sequential update order is considered. In this section we focus on the difference between synchronous and sequential update orders in situations in which those inequalities are not satisfied. We show that the behaviour is no longer identical and that sequential updating offers a more diverse behaviour.

First we study the synchronous update order.
\begin{thm}\label{t:hd:a}
Let $(a,b,c,d)\in \mathcal{P}$. Let us consider the (autonomous) evolutionary game $\varphi$ on $K_n$ with the utility function \eqref{e:utility}. Then there exists a nontrivial attractor of $\varphi$ if and only if either \eqref{e:Kn:0} or \eqref{e:Kn:1} hold.
\end{thm}

\begin{proof}
Obviously, if \eqref{e:Kn:0} or \eqref{e:Kn:1} hold, then either $(0,0,\ldots,0)$ or $(1,1,\ldots,1)$ are attractors, see Theorems \ref{t:Kn:0} and \ref{t:Kn:1}.

Now, let us assume that neither \eqref{e:Kn:0} nor \eqref{e:Kn:1} hold. Let $x\in S^V$ and $m=\sum_{i=1}^n x_i$. If we denote by $v_1(m)$ ($v_0(m)$) the utilities of vertices with $x_i=1$ ($x_i=0$), we can easily derive that
\begin{align}
v_1(m)-v_0(m) &= (m-1)a + (n-m) b - mc - (n-m-1)d \notag\\
		     &= -m ((c-a)+(b-d))+n(b-d)-(a-d).\label{e:dif:ut}
\end{align}

The simultaneous violation of \eqref{e:Kn:0} and \eqref{e:Kn:1} implies that $v_1(1)-v_0(1)\geq 0$ and $v_1(n-1)-v_0(n-1)\leq 0$. Then \eqref{e:dif:ut} implies that $(c-a)+(b-d)>0$ and that there exists $m^*\in \left[ 1,n-1\right]$ defined by
\begin{equation}\label{e:m*}
m^*:=\frac{n(b-d)-(a-d)}{(c-a)+(b-d)},
\end{equation}
such that $v_1(m^*)-v_0(m^*)=0$. Note that $\lim_{n\rightarrow \infty} \frac{m^*}{n} = \frac{d-b}{(a-c)+(d-b)}$, i.e.\ as $n$ tends to infinity $\frac{m^*}{n}$ tends to the interior fixed point of the replicator dynamics of HD and SH games (see Section 3).

We consider three invariant sets $\lbrace (0,0,\ldots,0) \rbrace$ and  $\lbrace(1,1,\ldots,1)\rbrace$ and $M^*$, which is given by
\[
M^*:= \Big\{ x\in S^V:\sum\limits_{i=1}^n x_i=m^* \Big\},
\]
and is nonempty if and only if $m^*\in \Nat$.  The sign of $v_1(m)-v_0(m)$ determines that
\begin{itemize}
	\item $\varphi_i(x)=1$ if $m<m^*$,
	\item $\varphi_i(x)=0$ if $m>m^*$,
	\item $\varphi_i(x)=x_i$ if $m=m^*$.
\end{itemize}

Therefore, the basins of attraction of these three invariant sets are given by
\begin{align*}
B \left(\left\lbrace(0,0,\ldots,0)\right\rbrace \right) &= \left\lbrace (0,0,\ldots,0) \right\rbrace \cup \Big\{ x\in S^V: m^*<\sum\limits_{i=1}^n x_i<n \Big\},\\
B \left(\left\lbrace(1,1,\ldots,1)\right\rbrace \right) &= \left\lbrace (1,1,\ldots,1) \right\rbrace \cup \Big\{ x\in S^V: 0<\sum\limits_{i=1}^n x_i<m^* \Big\},\\
B \left(M^* \right) &= M^*.
\end{align*}
Consequently, none of these invariant sets is an attractor (see Remark \ref{r:basin}) and there is only the trivial attractor $S^V$.
\end{proof}

The fact that $M^*$ is  not attractive may look surprising. The proof shows that this is caused by the synchronous update order which implies that all vertices update to full cooperation or full defection. Next we provide a characterization of the existence of an attractor for the sequential update order as well.

\begin{thm}\label{t:hd:na}
Let $(a,b,c,d)\in \mathcal{P}$. Let us consider a nonautonomous evolutionary game $\varphi$ on $K_n$ with the utility function \eqref{e:utility} and sequential update order. Then there exists a nontrivial attractor if and only if $(a,b,c,d)$ satisfy either

\begin{enumerate}[\upshape (a)]
\item inequality \eqref{e:Kn:0}, or
\item inequality \eqref{e:Kn:1}, or
\item $m^* \in [2,n-2]$ and $c-a+b-d>0$.
\end{enumerate}

\end{thm}

\begin{proof}
Using the notation of the proof of Theorem \ref{t:hd:a} we observe that there exists $\alpha, \beta\in\Real$ such that
\begin{equation}\label{e:affinity}
v_1(m)-v_0(m) = \alpha m + \beta.
\end{equation}

Hence, we can distinguish between five cases. Either
\begin{enumerate}[\upshape (i)]
\item $v_1(1)-v_0(1)<0$, or
\item $v_1(1)-v_0(1) \geq 0$ and $v_1(2)-v_0(2)<0$, or
\item $v_1(2)-v_0(2) \geq 0$ and $v_1(n-2)-v_0(n-2) \leq 0$, or
\item $v_1(n-2)-v_0(n-2) > 0$ and $v_1(n-1)-v_0(n-1) \leq0$, or
\item $v_1(n-1)-v_0(n-1)>0$.
\end{enumerate}

First, we can observe that (i) is satisfied if and only if \eqref{e:Kn:0} holds. Then Theorem \ref{t:Kn:0:na} implies that $(0,0,\ldots,0)$ is attractive. Similarly, (v) is satisfied if and only if \eqref{e:Kn:1} holds and Theorem \ref{t:Kn:0:na} yields that $(1,1,\ldots,1)$ is attractive.

In the remaining three cases (ii)-(iv), the equalities \eqref{e:affinity} and \eqref{e:dif:ut} imply that $c-a+b-d>0$, i.e.\ that there exists $m^* \in [-1,n-1]$ given by \eqref{e:m*} which satisfies $v_1(m^*)-v_0(m^*)=0$.

Let us consider case (iii) first. The inequalities (iii) imply that $m^* \in [2,n-2]$ and $c-a+b-d>0$. As in the proof of Theorem \ref{t:hd:a} we see immediately that $v_1(m)>v_0(m)$ if and only if $0<m<m^*$ and that $v_0(m)>v_1(m)$ if and only if $m^*<m<n$. Moreover, sequential update order implies that  for all $t\in\Nat_0$ we have $\upor(t)=\lbrace i \rbrace$ for some $i\in V$. This implies that
\begin{itemize}
\item if $0<m<m^*$, $\varphi_i(t+1,t,x) = 1$,
\item if $m^*<m<n$, $\varphi_i(t+1,t,x) = 0$,
\item if $m=m^*$, $\varphi_i(t+1,t,x) = x_i$.
\end{itemize}
Consequently, we distinguish between two cases
\begin{enumerate}[\upshape (a)]
\item if $m^*\in \Nat$. First, we consider $x\in S^V$ such that $1 \leq \sum\limits_{i=1}^n x_i=m < m^*$. Without loss of generality we can assume that the vertices are numbered so that $x_1=\dots=x_m=0$. Define
\begin{equation}\label{e:pr:attr}
A := \Big\{ (t,x)\in \Nat_0 \times S^V: \sum\limits_{i=1}^n x_i= m^* \Big\} .
\end{equation}
Then
\[
\varphi_i(i,0,x)=1,
\]
and thus there exists $x^*\in A$ (i.e.\ $\sum_{i=1}^n x_i^*=m^*$) so that
\[
\varphi(t,0,x) = x^*, \textrm{ for all } t\geq m.
\]
Similarly, if $x$ is such that $m^* < \sum_{i=1}^n x_i=m \leq n-1$, then the first $(m- m^*)$ vertices $i$ with $x_i=1$ switch to $x_i=0$.
Consequently, the set $A$ is the attractor of $\varphi$.

\smallskip
\item if $m^*\notin \Nat$, one could repeat the argument to get that any initial condition reaches a state with $\sum_{i=1}^n x_i=\lfloor m^* \rfloor$. If we are at such a state $x$ in time $t$ and $\upor(t)=\lbrace i \rbrace$ we have that
\[
\varphi_i(t+1,t,x)=1,
\]
independently of $x_i$ at time $t$. This implies that either $x$ remains unchanged or a state with $\sum_{i=1}^n x_i=\lceil m^* \rceil$ is reached.

Similarly, if we are at a state $x$ with $\sum_{i=1}^n x_i=\lceil m^* \rceil$ we have $\varphi_i(t+1,t,x)=0$, and either $x$ remains unchanged or a state with $\sum_{i=1}^n x_i=\lfloor m^* \rfloor$ is reached.

Consequently, we observe that the set
\begin{equation}\label{e:pr:cycle}
A = \Big\{ (t,x)\in \Nat_0 \times S^V: \sum\limits_{i=1}^n x_i= \lfloor m^* \rfloor \textrm{ or } \sum\limits_{i=1}^n x_i= \lceil m^* \rceil \Big\}
\end{equation}
is an attractor. The fact that the dynamical system always switches from a state with $\sum_{i=1}^n x_i=\lfloor m^* \rfloor$ to a state with $\sum_{i=1}^n x_i=\lceil m^* \rceil$ and the finiteness of the graph implies that $A$ is a union of cycles.
\end{enumerate}

To finish the proof, we consider cases (ii) and (iv), i.e.\ the situation in which $m^*\in [1,2) \cup (n-2,n-1]$ and $c-a+b-d>0$. In this case, the sets given by \eqref{e:pr:attr} and \eqref{e:pr:cycle} are invariant by the same argument as above. However, they are not attractive, since $(0,0,\ldots,0) \notin B(A)$ if $m^*\in[1,2)$ and $(1,1,\ldots,1) \notin B(A)$ if $m^*\in(n-2,n-1]$.
\end{proof}

\begin{figure}
\begin{minipage}{0.5\textwidth}
\begin{center}
\includegraphics[width=\textwidth]{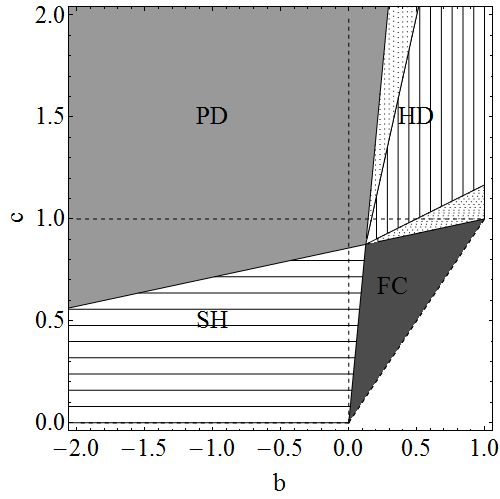}\\
\end{center}
\end{minipage}
\caption{Illustration of Theorem \ref{t:hd:na} with $a=1$ and $d=0$. Both $(0,0,\ldots,0)$ and $(1,1,\ldots,1)$ are attractive in the horizontally hatched region, $(0,0,\ldots,0)$ in the light gray region, $(1,1,\ldots,1)$ in the dark gray region. In the vertically hatched region, the attractive cycles \eqref{e:pr:cycle} and attractors \eqref{e:pr:attr} occur. There are no nontrivial attractors in the two dotted regions.}\label{f:attractors}
\end{figure}

\begin{rem}
To sum up, Theorems \ref{t:hd:a} and \ref{t:hd:na} provide a complete characterization of attractors of evolutionary games on $K_n$ with either synchronous or sequential update order. For synchronous update order, there are four possible outcomes:
\begin{itemize}
\item only $(0,0,\ldots,0)$ is attractive,
\item only $(1,1,\ldots,1)$ is attractive,
\item both $(0,0,\ldots,0)$ and $(1,1,\ldots,1)$ are attractive,
\item there is no nontrivial attractor.
\end{itemize}
These four regions correspond to those depicted in Figure \ref{f:att:reg}.

For sequential update order we have another two possibilities which consist of states in which both cooperators and defectors exist together:
\begin{itemize}
\item there is an attractive cycle given by \eqref{e:pr:cycle},
\item there is an attractive set of invariant states given by \eqref{e:pr:attr}.
\end{itemize}

Analyzing the above results, we see that the attracting cycle and the attracting set in which both cooperators and defectors exist together can occur if and only if $(a,b,c,d)\in  \mathcal{P}_{HD}\cup \mathcal{P}_{FC}$. $\mathcal{P}_{FC}$ is the only region in which all possible scenarios coexist, see Figure \ref{f:attractors}.

Again, note that Theorems \ref{t:hd:a} and \ref{t:hd:na} are consistent with the standard evolutionary game theory \cite{aHofbauer} as $n\To\infty$. Note that both the bistability region (both $(0,0,\ldots,0)$ and $(1,1,\ldots,1)$ are attractive) and the stable coexistence region in the sequential updating converge to $\mathcal{P}_{SH}$ and $\mathcal{P}_{HD}$, respectively. Note that, in finite populations, they are smaller but overreach to $\mathcal{P}_{FC}$ as well, see Figure \ref{f:attractors}.
\end{rem}

We provide a simple example to better illustrate the cycle of length $n(n+1)$ which has been constructed in the proof of Theorem \ref{t:hd:na}.
\begin{exmp}\label{x:cycle}
Let us assume that $(a,b,c,d)\in \mathcal{P}_{HD}\cup \mathcal{P}_{FC}$ and $m^*\in[2,n-2]$ is such that $m^*\notin \Nat$. Let us consider an evolutionary game on $K_n$ with utility function \eqref{e:utility} and sequential update order. We consider the initial condition:
\[
	x=(\overbrace{1,1,\ldots 1}^{\lfloor m^* \rfloor},\overbrace{0,0,\ldots,0}^{n-\lfloor m^* \rfloor}).
\]
Consequently, we derive that (bold numbers indicate the vertex which has just been updated)
\begin{align*}
	\varphi(1,0,x) &= (\mathbf{1},1,\ldots 1,0,0,\ldots,0),\\
	\varphi(2,0,x) &= (1,\mathbf{1},\ldots 1,0,0,\ldots,0),\\
	&\ldots  \\
	\varphi(\lfloor m^* \rfloor,0,x) &= (1,1,\ldots \mathbf{1},0,0,\ldots,0),\\
	\varphi(\lfloor m^* \rfloor+1,0,x) &= (1,1,\ldots 1,\mathbf{1},0,\ldots,0),\\
	\varphi(\lfloor m^* \rfloor+2,0,x) &= (1,1,\ldots 1,1,\mathbf{0},\ldots,0),\\
	&\ldots  \\
	\varphi(n,0,x) &= (1,1,\ldots 1,1,0,\ldots,\mathbf{0}),\\
	\varphi(n+1,0,x) &= (\mathbf{0},\underbrace{1,\ldots 1,1}_{\lfloor m^* \rfloor},0,\ldots,0).\\
\end{align*}
We repeat this argument to get that
\begin{align*}
	\varphi(2(n+1),0,x) &= (0,\mathbf{0},\underbrace{1,\ldots 1,1}_{\lfloor m^* \rfloor},0,\ldots,0),\\
	&\ldots  \\
	\varphi(p(n+1),0,x) &= (\underbrace{0,0,\ldots\mathbf{0}}_p,\underbrace{1,\ldots 1,1}_{\lfloor m^* \rfloor},0,\ldots,0),\\
	&\ldots  \\
	\varphi(n(n+1),0,x) &= (1,1,\ldots 1,0,0,\ldots,\mathbf{0}).
\end{align*}
\end{exmp}


\section{Irregular Graphs - Role of Utility Functions}\label{s:wheels}
In this section we study simple irregular graphs -- wheels $W_l$, $l\geq 4$, in which a central vertex is connected to all vertices of an $(l-1)$-cycle, see Figure \ref{f:wheels}. Our focus lies on identifying the importance of  different forms of utility functions, e.g.\ \eqref{e:utility}, \eqref{e:utility:M} or others.  Evolutionary games on regular graphs, which we have considered exclusively so far, are the same for aggregate $u_i^A$  and mean $u_i^M$ utility functions (see \eqref{e:utility}-\eqref{e:utility:M}), since $u_i^A$  is just a multiple of $u_i^M$ in this case. Straightforwardly, this is not longer true for irregular graphs.

\begin{figure}

\begin{minipage}{\textwidth}
\begin{center}
\begin{minipage}{0.35\textwidth}
\begin{center}
\includegraphics[width=\textwidth]{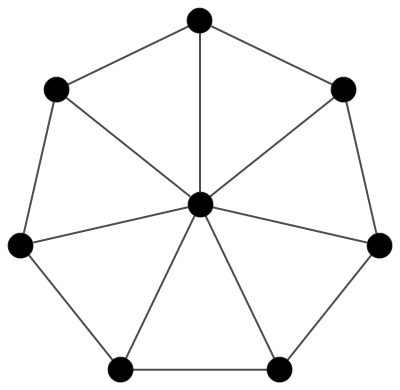}
\end{center}
\end{minipage}
\begin{minipage}{0.35\textwidth}
\begin{center}
\includegraphics[width=\textwidth]{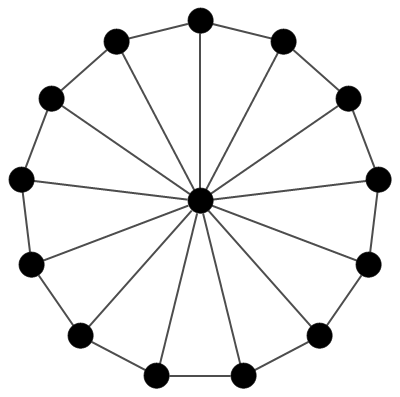}
\end{center}
\end{minipage}
\end{center}
\end{minipage}
\caption{Wheels $W_8$ and $W_{14}$.}\label{f:wheels}

\end{figure}

First, we study the attractivity of $(1,1,\ldots,1)$.

\begin{thm}
Let $(a,b,c,d)\in \mathcal{P}$. Let us consider an evolutionary game $\varphi$ on $W_l$. Then  $(1,1,\ldots,1)$ is an attractor if and only if
\begin{equation}\label{e:wheel}
\begin{cases}
c<\frac{2a+b}{l-1} & \textrm{ if } u_i^A \textrm{ is considered,}\\
c<\frac{2a+b}{3} & \textrm{ if } u_i^M \textrm{ is considered.}
\end{cases}
\end{equation}
\end{thm}

\begin{proof}
Let us number the vertices, so that the central vertex $1$ is connected to peripheral vertices $\lbrace 2,3,\ldots,n \rbrace$.
\begin{enumerate}
\item Let us consider the aggregate utility $u_i^A$ first and suppose again that the state $x\in S^V$ is such that $\sum_{i=1}^n x_i=n-1$.
	\begin{enumerate}
		\item If the central vertex is the single defector, then
			\begin{eqnarray*}
				u_1^A&=&(l-1) c,\\
				u_i^A&=&2a+b, \quad i=2,3,\ldots,n.
			\end{eqnarray*}
			Obviously, the first inequality in \eqref{e:wheel} is equivalent to $u_1^A<u_i^A$.
		\item If the  single defector is a peripheral vertex, we can, without loss of generality, assume that $x=(1,0,1,\ldots, 1)$, i.e.\ vertex $2$ is the single defector. Then, we have that
			\begin{eqnarray*}
				u_1^A&=& (l-2)a+b,\\
				u_2^A&=& 3c,\\
				u_i^A&=& 2a+b, \quad i\in N_1(2) \cap \lbrace 2,3,\ldots,n \rbrace,\\
				u_i^A&=& 3a, \quad i\notin N_1(2).
			\end{eqnarray*}
			Then we can bound $u_2^A$ by $u_1^A$ from above
			\[
				u_2^A=3c \overset{\eqref{e:wheel}}{<} 3\frac{2a+b}{l-1} \overset{\textrm{(A5)}}{\leq} 3\frac{(l-2)a+b}{3} = u_1^A.
			\]
			Consequently $\varphi(x)=(1,1,\ldots,1)$.
	\end{enumerate}
\item If the mean utility $u_i^M$ is considered instead then
\begin{enumerate}
		\item If the central vertex is the single defector, we have
			\begin{eqnarray*}
				u_1^M&=& c,\\
				u_i^M&=&\frac{2a+b}{3}, \quad i=2,3,\ldots,n.
			\end{eqnarray*}
			The latter inequality in \eqref{e:wheel} is equivalent to $u_1^M<u_i^M$.
		\item If the  single defector is a peripheral vertex, we can again, without loss of generality, assume that $x=(1,0,1,\ldots, 1)$. Then, the utilities have the form
			\begin{eqnarray*}
				u_1^M&=& \frac{(l-2)a+b}{l-1},\\
				u_2^M&=& c,\\
				u_i^M&=& \frac{2a+b}{3}, \quad i\in N_1(2) \cap \lbrace 2,3,\ldots,n \rbrace,\\
				u_j^M&=& a, \quad j\notin N_1(2),
			\end{eqnarray*}
			which immediately implies that $u_2^M<u_i^M$, for all $i\in N_1(2)$.
	\end{enumerate}
\end{enumerate}
Paragraphs 1(a) and 2(a) show that both inequalities in \eqref{e:wheel} are also necessary. If they are violated, then either $\varphi(x)=x$ (if equalities hold) or $\varphi_i(x)=0$ for all $i$ (if reverse inequalities hold), i.e., all vertices switch to defection and $(1,1,\ldots,1)$ cannot be attained. 
\end{proof}

We can make a few straightforward observations.
\begin{rem}
The proof could be repeated for nonautonomous evolutionary games with sequential update order.

Note that, (A5) was used in the proof. If we do not assume that (A5) holds, i.e.\ $a$ could be non-positive, then the condition for the aggregate utility would be
\[
c<\min \left\lbrace \frac{2a+b}{l-1}, \max\left\lbrace\frac{2a+b}{3},\frac{(l-2)a+b}{3} \right\rbrace \right\rbrace.
\]

In \eqref{e:wheel} the former inequality implies the latter. The aggregate utility function favours the vertex with higher degree, whereas the mean utility function eliminates differences resulting from different degrees. More importantly, if we consider the mean utility function the necessary and sufficient condition is independent of the wheel size $l$, whereas with the aggregate utility, the inequality is satisfied only for small wheels. Indeed, we can rewrite the inequality  as $l<1+\frac{2a+b}{c}$.

Note that the inequalities in \eqref{e:wheel} can be satisfied if and only if $c<a$, i.e.\ $(a,b,c,d)\in \mathcal{P}_{SH} \cap\mathcal{P}_{FC}$.

\end{rem}

We can simply formulate a similar result for full defection $(0,0,\ldots,0)$.

\begin{thm}
Let $(a,b,c,d)\in \mathcal{P}$. Let us consider an evolutionary game $\varphi$ on $W_l$. Then $(0,0,\ldots,0)$ is an attractor if and only if
\begin{equation}\label{e:wheel:0}
\begin{cases}
a<\frac{2d+c}{l-1}  & \textrm{ if } u_i^A \textrm{ is considered,}\\
a<\frac{2d+c}{3} & \textrm{ if } u_i^M \textrm{ is considered.}
\end{cases}
\end{equation}
\end{thm}


\section{Open questions}\label{s:oq}

In this paper we define evolutionary games on graphs rigorously as dynamical systems and also state several results on the existence of attractors, their basins of attraction and their relationship to update orders and regularity. Our results lead to many open questions, we list those which we find most interesting to consider as a next step towards the development of a theory of evolutionary games:

\begin{enumerate}[(A)]
\item \textbf{Mixed fixed points}: Find sufficient conditions on the graph and admissible parameters $(a,b,c,d)\in\mathcal{P}$ which ensure existence/nonexistence of mixed fixed points $x^*=(x^*_1,\ldots,x^*_n)$, in which cooperators $x_i^*=1$ and defectors $x_j^*=0$ coexist for some $i\neq j$.

\item \textbf{Attractors}: Construct efficient methods for finding all attractors for a given graph and admissible parameters $(a,b,c,d)\in\mathcal{P}$.

\item \textbf{Maximal number of fixed points}: Find the maximal number of fixed points for all connected graphs with $n$ vertices.

\item \textbf{Realization of mixed fixed points}: Determine all admissible parameters $(a,b,c,d)\in\mathcal{P}$ for which there exists an evolutionary game on a connected graph with a mixed fixed point.
\item \textbf{Existence of cycles}: Determine all admissible parameters $(a,b,c,d)\in\mathcal{P}$ for which there exists (or does not exist) a cycle (of length at least 2) of an evolutionary game with sequential/synchronous update orders.
\item \textbf{Maximal cycle}: Find the maximal length of a cycle of an evolutionary game on an arbitrary graph with $n$ vertices.
\item \textbf{Graph properties and evolutionary games}: Relate graph features (size, regularity, diameter/girth,  connectivity, clique number etc.) to the properties of evolutionary games on these graphs (existence of attractors, fixed points, cycles, ...).
\item \textbf{Different dynamics}: In this paper we used imitation dynamics \eqref{e:imitation}. Describe major differences in the case that different deterministic dynamics are used (e.g.\ birth-death, death-birth, see Remark \ref{r:evgame}).
\item \textbf{Utility functions}: In Section \ref{s:wheels} we showed that the aggregate utility function \eqref{e:utility} favours vertices with higher degree. Describe this phenomenon precisely and analyze the role of other utility functions. See \cite{aMaciejewski} for the discussion on averaging and accumulation of utility functions in stochastic evolutionary games.
\item \textbf{Non-omitting update orders}: Theorem \ref{t:Kn:0:na} and Example \ref{x:Hella:2} show that conditions \eqref{e:Kn:0} and \eqref{e:Kn:1} are only sufficient for attractivity of $(0,0,\ldots,0)$ and $(1,1,\ldots,1)$. Find necessary and sufficient conditions for any non-omitting update order.
\item \textbf{Regular graphs}: Example \ref{x:Cayley} showed that \eqref{e:kreg:0} and \eqref{e:kreg:1} are only sufficient but not necessary for attractivity of $(0,0,\ldots,0)$ and $(1,1,\ldots,1)$. Construct similar counterexamples for arbitrary $k$. Find a necessary and sufficient condition for attractivity of $(0,0,\ldots,0)$ and $(1,1,\ldots,1)$ on $k$-regular graphs.
\item \textbf{Irregular graphs}: Identify features of irregular graphs that play an essential role in the dynamics of evolutionary games on them.
\end{enumerate}

\section*{Acknowledgements}
We thank an anonymous referee for comments which lead to an improvement of the paper.
We thank Hella Epperlein for her contribution of Example \ref{x:Hella:2}. This work is partly supported by the German Research Foundation (DFG) through the Cluster of Excellence 'Center for Advancing Electronics Dresden' (cfaed). The last author acknowledges the support by the Czech Science Foundation, Grant No. 201121757.



\begin{thebibliography}{9}
\bibitem{aAllen}
B.~Allen, A.~Traulsen, C.~Tarnita, M.~A.~Nowak, \emph{How mutation affects evolutionary games on graphs}, J. Theor. Biol. {\bf 299} (2012), 97–-105.

\bibitem{aAllenNowak}
B.~Allen, M.~A.~Nowak, \emph{Games on Graphs}, EMS Surveys in Mathematical Sciences {\bf 1} (2004), 113--151.

\bibitem{aBroom}
M.~Broom, C. Hadjichrysanthou, J. Rycht\'{a}\v{r}, \emph{Evolutionary games on graphs and the speed of the evolutionary process}, Proc. R. Soc. A {\bf 466} (2010), 1327--1346.

\bibitem{aChen}
Y.-T.~Chen, \emph{Sharp benefit-to-cost rules for the evolution of cooperation on regular graphs}, The Annals of Applied Probability {\bf 23} (2013), 637-–664. 

\bibitem{bCox}
J.~T.~Cox, R.~Durrett, E.~A.~Perkins, \emph{Voter Model Perturbations and Reaction Diffusion Equations}, American Mathematical Society, 2013.

\bibitem{aDebarre}
F.~D\'{e}barre, C.~Hauert and M.~Doebeli, \emph{Social evolution in structured populations,} Nature communications {\bf 5} (2014), Article No. 3409.

\bibitem{2003:Devaney}
R.~L.~Devaney, \emph{An Introduction to Chaotic Dynamical Systems,}  Reprint of the second (1989) edition,
Westview Press, Boulder, 2003.


\bibitem{bDixit} A. Dixit, S. Skeath and D. Reiley, \emph{Games of strategy,} W. W. Norton, New York, 2009.

\bibitem{aGodsil}
C.~Godsil and G.~Royle, \emph{Algebraic Graph Theory,} 2nd edition, Springer-Verlag, New York, 2001.


\bibitem{aBroom2}
C.~Hadjichrysanthou, M.~Broom,  J. Rycht\'{a}\v{r}, \emph{Evolutionary Games on Star Graphs Under Various Updating Rules}, Dynamic Games and Applications {\bf 1} (2011), 386--407.

\bibitem{aHauert}
C.~Hauert and M.~Doebeli, \emph{Spatial structure often inhibits the evolution of cooperation in the snowdrift game,} Nature {\bf 428} (2004), 643--646.

\bibitem{aHungerford}
T.~W.~Hungerford, \emph{Algebra,} Springer-Verlag, New York, 1980.

\bibitem{aHofbauer}
J.~Hofbauer, K.~Sigmund, \emph{Evolutionary Game Dynamics,} Bulletin of the American Mathematical Society {\bf 40} (2003), 479--519.

\bibitem{bHofbauer}
J.~Hofbauer, K.~Sigmund, \emph{The Theory of Evolution and Dynamical Systems,} Cambridge University Press, 1988.

\bibitem{aKandori}
M.~Kandori, G.~Mailath and R.~Rob, \emph{Learning, mutation and long run equilibria in games,} Econometrica {\bf 61} (1993), 29--56.

\bibitem{aLagunoff}
R.~Lagunoff, A.~Matsui, \emph{Asynchronous Choice in Repeated Coordination Games,} Econometrica {\bf 65} (1997), 1467--77.

\bibitem{aLibich}
J.~Libich, \emph{A Note on the Anchoring Effect of Explicit Inflation Targets,} Macroeconomic Dynamics {\bf 13} (2009), 685--697.

\bibitem{aLibich2}
J.~Libich, P.~Stehl\'{\i}k \emph{Monetary policy facing fiscal indiscipline under generalized timing of actions}, Journal of Institutional and Theoretical Economics {\bf 168} (2012), 393--431

\bibitem{aLieb}
E.~Lieberman, C.~Hauert and M.~A.~Nowak, \emph{Evolutionary dynamics on graphs,} Nature {\bf 433} (2005), 312--316.

\bibitem{aMaciejewski}
W.~Maciejewski, F.~Fu, C.~Hauert, \emph{Evolutionary Game Dynamics in Populations with Heterogeneous Structures,} PLoS Computational Biology {\bf 10} (2014), Article No. e1003567.

\bibitem{aMaynard}
J.~Maynard Smith, \emph{The theory of games and the evolution of animal conflicts,} Journal of Theoretical Biology {\bf 47} (1974), 209--221.

\bibitem{wMcKay}
B.~D.~ McKay, \emph{Graph6 and sparse6 graph formats}. {\tt http://cs.anu.edu.au/\char`~ bdm/data/formats.html}.

\bibitem{bMyerson}
R.~Myerson, \emph{Game Theory: Analysis of Conflict,} Harvard University Press, Cambridge, 1997.

\bibitem{aNash}
J.~F.~Nash, \emph{The Bargaining Problem,} Econometrica, {\bf 18} (1950), 155--162.

\bibitem{bNowak}
M.~A.~Nowak, \emph{Evolutionary Dynamics: Exploring the Equations of Life,} Harvard University Press, Cambridge, 2006.

\bibitem{aNowak}
M.~A.~Nowak, R.~M. May, \emph{Evolutionary games and spatial chaos,} Nature {\bf 359} (1992), 826--829.

\bibitem{aOhtsuki}
H.~Ohtsuki, M.~A.~Nowak, \emph{Evolutionary games on cycles,} Proc. R. Soc. B {\bf 273} (2006), 2249--2256.

\bibitem{aOhtsuki2}
H.~Ohtsuki, M.~A.~Nowak, \emph{Evolutionary stability on graphs,} Journal of Theoretical Biology {\bf 251}, 698--707.

\bibitem{aSzabo}
G.~Szab\'{o}, G.~F\'{a}th, \emph{Evolutionary games on graphs,} Phys. Rep. {\bf 446} (2007), 97--216.

\end{thebibliography}
\end{document}